\documentclass[11pt]{article}
\usepackage{amsmath}
\usepackage{amsfonts}
\usepackage{amssymb}

\usepackage{amssymb}
\usepackage{amsthm}
\usepackage{graphicx}              
\usepackage{amsmath}               
\usepackage{amsfonts}              
\usepackage{amsthm}                
\usepackage{setspace}
\usepackage{epstopdf}
\usepackage{epsfig}

\newtheorem{Theorem}{Theorem}[section]
\newtheorem{Proposition}[Theorem]{Proposition}
\newtheorem{Lemma}[Theorem]{Lemma}
\newtheorem{Corollary}[Theorem]{Corollary}
\newtheorem{Definition}{Definition}[section]

\makeatletter

\newcommand{\Rmnum}[1]{\expandafter\@slowromancap\romannumeral #1@}
\makeatother
\newtheorem{remark}{Remark}[section]
\allowdisplaybreaks
\numberwithin{equation}{section}

\title{\Large\bf\boldmath
Asymptotic decomposition for nonlinear damped Klein-Gordon equations}

\author{\large \large Ze Li \qquad  Lifeng Zhao }

\date{}

\begin{document}

\maketitle

{\bf{Abstract } }
In this paper, we proved that if the solution to damped focusing Klein-Gordon equations is global forward in time, then it will decouple into a finite number of equilibrium points with different shifts from the origin. The core ingredient of our proof is the existence of the ``concentration-compact attractor" which yields a finite number of profiles. Using damping effect, we can prove all the profiles are equilibrium points.


\section{Introduction}\label{se1}

In this paper, we consider the following damped focusing Klein-Gordon equation:
\begin{align}\label{1}
\left\{ \begin{array}{l}
 {u_{tt}} - \Delta u + u +2\alpha {u_t} - {\left| u \right|^{p - 1}}u = 0, \\
 u(0) = {u_0},\mbox{   }{\partial _t}u(0) = {u_1}\in \mathcal{H}, \\
 \end{array} \right.
\end{align}
where $\mathcal{H}=H^1(\Bbb R^d)\times L^2(\Bbb R^d)$, $\alpha\ge0$. The energy is given by
$$E(f,g) = \int_{{\Bbb R^d}} {\left( {\frac{1}{2}{{\left| {\nabla f} \right|}^2} + \frac{1}{2}{{\left| f \right|}^2} + \frac{1}{2}{{\left| g \right|}^2} - \frac{1}{{p + 1}}{{\left| f \right|}^{p + 1}}} \right)} dx.$$

 Dispersive equations such as Klein-Gordon equations, wave equations, Schr\"odinger equations have been intensively studied for decades. For $\alpha=0$, namely nonlinear Klein-Gordon equation, T. Cazenave \cite{Ca} gave the following dichotomy: solutions either blow up at finite time or are global forward in time and bounded in $\mathcal{H}$, provided $1<p<\infty$, when $d=1,2$ and $1<p<\frac{d}{d-2}$ if $d\ge3$. For $\alpha>0$, E. Feireisl \cite{F} gave an independent proof of the boundedness of the trajectory to global solutions, for $1<p<1+\min(\frac{d}{d-2},\frac{4}{d})$ when $d\ge3$, and in his paper \cite{F2}, the case $d=1$ is considered. N. Burq, G. Raugel, W. Schlag \cite{BRS} studied the long time behaviors of solutions to nonlinear damped Klein-Gordon equations in radial case. They proved that radial global solutions will converge to equilibrium points as time goes infinity. A natural problem is what happens for non-radial solutions? It is widely conjectured that the solutions will decouple into the superposition of equilibrium points. A positive result given by  E. Feireisl \cite{F} implied there exists a global solution which decouples into a finite number of equilibrium points with different shifts from origin. Indeed, this problem is closely related to the soliton resolution conjecture in dispersive equations. The ( imprecise sense ) soliton resolution conjecture states that for ``generic" large global solutions, the evolution asymptotically decouples into the superposition of divergent solitons, a free radiation term, and an error term tending to zero as time goes to infinity. For more expression and history, see A. Soffer \cite{S}.

 There are a lot of works devoted to the verification of the soliton resolution conjecture. T. Duyckaerts, C. Kenig, and F. Merle \cite{DKM} first make a breakthrough on this topic. For radial data to three dimensional focusing energy-critical wave equations, they proved the solution with bounded trajectory is in fact a superposition of a finite number of rescalings of the ground state plus a radiation term which is asymptotically a free wave. One of the key ingredient of their arguments is the novel tool, called ``channels of energy" introduced by \cite{DKM} \cite{DKM2}. The method developed by them has been applied to many other situations, such as \cite{CKLS1} \cite{CKLS2} \cite{KLS} \cite{KLLS3} \cite{KLLS2} for wave maps, \cite{CKLS3} \cite{DKM3} \cite{JLX} \cite{R} for semilinear wave equations. By a weak version of outer energy inequality, the soliton resolution along a sequence of times was proved by R. Cote, C. Kenig, A. Lawrie and W. Schlag \cite{CKLS3} for four dimensional critical wave equations in radial case, by R. Cote \cite{C} for equivariant wave maps, and by H. Jia, and C. Kenig \cite{JK} for semilinear wave equations, wave maps.

 It is known that (\ref{1}) admits a radial positive stationary solution with the minimized energy among all the non-zero stationary solutions.
 Besides the ground state, (\ref{1}) also has an infinite number of nodal solutions which owns zero points. (see for instance H. Berestycki, and P.L. Lions \cite{BL}). Hence it seems that subcritical problems need different techniques. The dynamics of solutions below and slightly above the ground state is known. If $\alpha=0$, for initial data with energy below the ground state, I.E. Payne, and D.H. Sattinger \cite{PS} proved that the solution either blows up in finite time or scatters to zero. K. Nakanishi, and W. Schlag \cite{NS} described the asymptotics of the solutions with energy slightly larger than the ground state. In fact they proved the trichotomy forward in time: the solution (1) either blows up at finite time (2) or globally exists and scatters to zero (3) or globally exists and scatters to the ground states. In radial setting, the above trichotomy was obtained in K. Nakanishi, and W. Schlag \cite{NS1}, followed by K. Nakanishi, and W. Schlag \cite{NS2} in non-radial case. The main technical ingredient of their papers is the ``one pass" theorem which excludes the existence of (almost) homoclinic orbits between the ground state and (almost) heteroclinic orbits connecting ground state $Q$ with $-Q$. N. Burq, G. Raugel, and W. Schlag \cite{BRS} studied the longtime dynamics for damped Klein-Gordon equations in radial case. By developing some dynamical methods especially invariant manifolds, they proved the $\omega$-limit set of the trajectory is just one single point, hence they showed the dichotomy in forward time (1) the solution either blows up at finite time, (2) or converges to some equilibrium point.

In this paper, we aim to study the long time behaviors of damped Klein-Gordon equations without radial assumptions.
Let
$$\lambda (d) = \left\{ \begin{array}{l}
 \infty ,\mbox{   }\mbox{  }\mbox{   }\mbox{  }\mbox{  }\mbox{   }\mbox{  }\mbox{   }d = 1,2 \\
 1 + \frac{4}{{d - 2}},\mbox{  }d = 3,4.
 \end{array} \right.
$$
Then, we have the main theorem as follows:
\begin{Theorem}
Let $\alpha>0$, $1\le d\le 4$, $1<p<\lambda(d)$. For any data $(u_0,u_1)\in\mathcal{H}$. Then\\
$(i)$ either the solution of (\ref{1}) blows up at finite positive time,\\
$(ii)$ or it is global forward in time with unbounded trajectory;\\
$(iii)$ or for any time sequence $t_n\to\infty$, up to a subsequence, there exist $0\le J<\infty$, $x_{j,n}\in \Bbb R^d$ for $j=1,2,...,J$ and equilibrium points $\{Q^j\}$ such that
$$u(t_n) = \sum\limits_{j = 1}^J {Q^j}(x - x_{j,n}) + o_{H^1}(1),$$
and $\mathop {\lim }\limits_{t \to \infty } {\partial _t}u(t) = 0$, in $L^2$, where $\{x_{j,n}\}$ satisfies the separation property:
$$
\mathop {\lim }\limits_{n \to \infty } \left| {{x_{j,n}} - {x_{i,n}}} \right| = 0, \mbox{  }{\rm{for}}\mbox{  }i\neq j.
$$
\end{Theorem}

An adaptation of arguments in T. Cazenave \cite{Ca} shows for $1<p<\infty$, when d=1,2, $1<p<\frac{d}{d-2}$, when $d\ge 3$, every global solution has bounded trajectory. Therefore, for this range of $p$, we have the following dichotomy:
\begin{Corollary}\label{lize}
For $1<p<\infty$, when d=1,2, $1<p<\frac{d}{d-2}$, when $d\ge 3$, either the solution to (\ref{1}) blows up at finite time $((i)$ in Theorem 1.1$)$ or decouples into the superposition of equilibriums $((iii)$ in Theorem 1.1$)$.
\end{Corollary}

\begin{remark}\label{lize2}
For $3\le d\le 6$, $1<p<\frac{d}{d-2}$, the proof in E. Feireisl \cite{F} is sufficient to give Corollary \ref{lize}. Thus, in this paper, when $d\ge 3$, we always assume $p\ge \frac{d}{d-2}$.
\end{remark}

In order to describe our proof, the following notions are needed:
\begin{Definition}
Given any $h\in \Bbb R^d$, let $\tau_h:\mathcal{H}\to\mathcal{H}$ be the shift operator $\tau_h f(x)=f(x-h)$, and we denote the translation group by $G=\{\tau_h:h\in\Bbb R^d\}$. Given any $K\subseteq\mathcal{H},$ we denote the orbit of $K$ by $GK=\{gf:g\in G, f\in K\}$. If $GK=K$, then we call $K$ G-invariant. Suppose that $J\ge0$ is an integer, we let
$$JK \equiv \left\{ {{f_1} + ... + {f_J}:{f_1},{f_2},...,{f_J} \in K} \right\}.$$
We say $E\subseteq\mathcal{H}$ is G-precompact with $J$ components if $E\subseteq J(GK)$ for some compact $K\subseteq\mathcal{H}$ and $J\ge1$.
\end{Definition}
Our proof is divided into three parts. In the first step, we prove the trajectory of $u(t)$ is attracted by a G-precompact set with $J$ components, namely the existence of concentration-compact attractor. The key ingredient in this step is frequency localization and spatial localisation. The idea of ``concentration compact" attractor was introduced by T. Tao \cite{TT}. In the second step, for any sequence going to infinity, we prove up to a subsequence there exist a finite number profiles. Then by applying perturbation theorem, we obtain a nonlinear profile decomposition. Using damping effect of (\ref{1}), we can show all the profiles are exactly equilibriums. Finally we prove the convergence for all time.

Our paper is organized as follows: In Section 2, we recall some preliminaries, such as Strichartz estimates, local wellposedness, perturbation theorem. In Section 3, we prove the frequency localization and spatial localization. In Section 4, we prove the existence of concentration-compact attractor.  In Section 5, we extract the profiles and finish our proof by using damping.

\vskip 0.2in

\noindent{\bf Notation and Preliminaries  }
We will use the notation $X\lesssim Y$ whenever there exists some positive constant $C$ so that $X\le C Y$. Similarly, we will use $X\sim Y$ if $X\lesssim Y \lesssim X$.
 We define the Fourier transform on $\mathbb{R}^d$ to be $$F(f)(\xi) =\tfrac1{(2\pi)^d} \int_{\mathbb{R}^d} e^{- ix\cdot\xi} f(x)\,\mathrm{d}x,$$
 $P_N$ is the usual Littlewood-Paley decomposition operator with frequency truncated in $N$. Similarly, we use $P_{\le N}$ and $P_{\ge N}$. Sometimes, we denote $P_{<\mu}u$ by $u_{<\mu}$. $\|u(t)\|_\mathcal{H}$ means $\|(u(t),\partial_t u(t))\|_{\mathcal{H}}$.
All the constants are denoted by $C$ and they can change from line to line.

\section{Preliminaries}
As explained in Remark \ref{lize2}, we only need to consider 
$$\left\{ \begin{array}{l}
 1 < p < \infty ,\mbox{  }\mbox{  }\mbox{  }\mbox{  }\mbox{  }\mbox{  }\mbox{  }\mbox{  }\mbox{  }\mbox{  }d = 1,2 \\
 \frac{d}{{d - 2}} \le p < 1 + \frac{4}{{d - 2}},d = 3,4.
 \end{array} \right.
$$

In this section we give the Strichartz estimates, local wellposedness and perturbation theorem, we closely follow notations in \cite{BRS}.
Consider the linear equation,
\begin{align}\label{2.1}
u_{tt}+2\alpha u_t-\Delta u+u=G, \mbox{  }\big(u(0,x),u_t(0,x)\big)=(u_0,u_1)\in \mathcal{H},
\end{align}
then by Duhamel principle,
\begin{align*}
 u(t) &= {e^{ - \alpha t}}\left[ {\cos (t\sqrt { - \Delta  + 1 - {\alpha ^2}} ) + \alpha \frac{{\sin (t\sqrt { - \Delta  + 1 - {\alpha ^2}} )}}{{\sqrt { - \Delta  + 1 - {\alpha ^2}} }}} \right]{u_0} \\
  &+ {e^{ - \alpha t}}\frac{{\sin (t\sqrt { - \Delta  + 1 - {\alpha ^2}} )}}{{\sqrt { - \Delta  + 1 - {\alpha ^2}} }}{u_1} + \int_0^t {\frac{{\sin ((t - s)\sqrt { - \Delta  + 1 - {\alpha ^2}} )}}{{\sqrt { - \Delta  + 1 - {\alpha ^2}} }}} {e^{ - \alpha (t - s)}}G(s)ds \\
  &\triangleq {S_{1,\alpha}}(t){u_0} + {S_{2,\alpha}}(t){u_1} + \int_0^t {{S_{2,\alpha}}} (t - s)G(s)ds.
 \end{align*}
Define $S_L(t)(u_0,u_1)=S_{1,\alpha}u_0+S_{2,\alpha}u_1$.

The Strichartz estimates are given by the following lemma. We emphasize that since we only need local Strichartz estimates in this paper, it is possible to get $L^p_t(I;L^q_x)$ estimates for non-admissible pair by H\"older inequality (see (\ref{wang}) below). 
\begin{Lemma}(\cite{BRS})
Let $t_0>T>0$, $u$ be a solution to (\ref{2.1}) on $[t_0-T,t_0+T]\times\Bbb R^d$. For $d\ge3$, $\theta^*=\frac{d+2}{d-2}$, we have
$$
\mathop {\sup }\limits_{t \in [-T,T]} {\left\| {(u(t),\partial_tu(t))} \right\|_{\mathcal{H}}} + {\left\| u \right\|_{L_t^{{\theta ^ * }}([-T,T];L_x^{2{\theta ^ * }})}} \lesssim e^{T\alpha}\left[ {{{\left\| {({u_0},{u_1})} \right\|}_\mathcal{H}} + \int_{-T}^T {{{\left\| {G(s)} \right\|}_2}ds} } \right].
$$
If $d=3,4$, then
$${\left\| u \right\|_{L_t^{\frac{4}{{d - 2}}}([ - T,T];L_x^{\frac{{4d}}{{d - 2}}})}} \lesssim {e^{T\alpha }}\left[ {{{\left\| {({u_0},{u_1})} \right\|}_{\cal H}} + \int_{ - T}^T {{{\left\| {G(s)} \right\|}_2}ds} } \right].
$$
If $d=1,2$, it holds that
$$\mathop {\sup }\limits_{t \in [-T,T]} {\left\| {(u(t),\partial_tu(t))} \right\|_{\mathcal{H}}} \lesssim e^{T\alpha}\left[ {{{\left\| {({u_0},{u_1})} \right\|}_\mathcal{H}} + \int_{-T}^T {{{\left\| {G(s)} \right\|}_2}ds} } \right].
$$
Define $Q=p+1$, when $d=1,2$, $p<Q<\frac{2p}{p(d-2)-d}$, when $3\le d\le 4$, $\frac{d}{d-2}\le p<1+\frac{4}{d-2}$, we have
\begin{align}\label{wang}
{\left\| u \right\|_{L_t^Q([ - T,T];L_x^{2p})}}\lesssim C(T)\left[ {{{\left\| {({u_0},{u_1})} \right\|}_{\cal H}} + \int_{ - T}^T {{{\left\| {G(s)} \right\|}_2}ds} } \right].
\end{align}
\end{Lemma}

As a corollary of Strichartz estimates, we have the perturbation theorem.
\begin{Lemma}
Let $M>0$. there exists $\varepsilon_0=\varepsilon_0(M)$ satisfying the following: Let $I\subset \Bbb R^+$ is a finite interval containing $t_0$, $\tilde{u}$ is defined on $I\times\Bbb R^d$, and satisfies
$$
\mathop {\sup }\limits_{t \in I} {\left\| {(\tilde{u},{\partial _t}\tilde{u})(t)} \right\|_\mathcal{H}}\le M.
$$
Suppose that $v$ is a solution to (\ref{1}) with initial data $(v(t_0),\partial_tv(t_0))$ at time $t_0$. Let $\varepsilon\in (0,\varepsilon_0)$, suppose that
\begin{align*}
\partial_{tt}\tilde{u}+2\alpha \partial_t \tilde{u}-\Delta \tilde{u} +\tilde{u}-|\tilde u|^{p-1}\tilde{u}&=e,\\
\|e\|_{L^1_t(I;L^2_x)}+\|S_{1,\alpha}(t-t_0)(\tilde{u}-v)(t_0)\|_{L^{Q}_t(I;L^{2p}_x)}+\|S_{2,\alpha}(t-t_0)
(\partial_t\tilde{u}-\partial_tv)(t_0)\|_{L^{Q}_t(I;L^{2p}_x)}&\le \varepsilon.
\end{align*}
Then
$$\|\tilde u-v-S_L(t-t_0)(\tilde{u}-v)(t_0)\|_{L^{\infty}_t(I;\mathcal{H})}+\|\tilde u-v\|_{L^{Q}_t(I;L^{2p}_x)}\le C(M)\varepsilon.
$$
\end{Lemma}

Combining T. Cazenave \cite{Ca} and N. Burq, G. Raugel, W. Schlag \cite{BRS}, we obtain the local wellposedness theorem as follows:
\begin{Proposition}\label{local}
For $(u_0,u_1)\in \mathcal{H}$, there exists $T>0$ such that (\ref{1}) is well-defined in $[0,T)$, with $T$ depending on $\|(u_0,u_1)\|_{\mathcal{H}}$. Furthermore, if $\|(u_0,u_1)\|_{ \mathcal{H}}<\epsilon$ with $\epsilon$ sufficiently small, then there exits $\gamma>0$ such that
$$\|(u(t),\partial_tu(t))\|_{\mathcal{H}}\le Ce^{-\gamma t}\|(u_0,u_1)\|_{\mathcal{H}}.$$ Moreover, if the solution $u(t)$ is globally defined, then we have
$$
\int^{\infty}_0\|\partial_tu(s)\|_2^2ds<\infty.
$$
\end{Proposition}

\section{Frequency localization and Spatial localization}
Since we focus on bounded solution throughout the paper, we assume
$$\mathop {\sup }\limits_{t \in [0,\infty)} {\left\| {(u,{\partial _t}u)(t)} \right\|_\mathcal{H}} \le E.$$
In the first step, we prove the localization of frequency, namely
\begin{Lemma}\label{13}
For any $\mu_0>0$ there exists $c(\mu_0)>0$ depending on $E$ such that
\begin{align}
&\mathop {\lim \sup }\limits_{t \to \infty } \|P_{\ge \frac{1}{c(\mu_0)}}u(t)\|_{H^1}\le \mu_0, \label{345}\\
&\mathop {\lim \sup }\limits_{t \to \infty } \|P_{\ge \frac{1}{c(\mu_0)}}\partial_tu(t)\|_{L^2}\le \mu_0. \label{kj}
\end{align}
\end{Lemma}
\begin{proof}
From Duhamel principle,
\begin{align*}
{P_{ \ge \frac{1}{{{\mu}}}}}u(t) = {S_{1,\alpha }}{P_{ \ge \frac{1}{{{\mu}}}}}{u_0} + {S_{2,\alpha }}{P_{ \ge \frac{1}{{{\mu}}}}}{u_1} + \int_0^t {{S_{2,\alpha }}} (t - s){P_{ \ge \frac{1}{{{\mu}}}}}\left( {{{\left| u \right|}^{p - 1}}u} \right)(s)ds.
\end{align*}
Since
\begin{align*}
{\left\| {{S_{1,\alpha }}{P_{ \ge \frac{1}{{{\mu}}}}}{u_0}} \right\|_{{H^1}}}{\rm{ }} \le {e^{ - \alpha t}}{\left\| {{u_0}} \right\|_{{H^1}}},\mbox{  }\mbox{  }\mbox{  }{\left\| {{S_{2,\alpha }}{P_{ \ge {\mu}^{ - 1}}}{u_1}} \right\|_{{H^1}}}{\rm{ }} \le {e^{ - \alpha t}}{\left\| {{u_1}} \right\|_{{L^2}}},
\end{align*}
for $\mu$ sufficiently small, we have
\begin{align*}
&{\left\| {{P_{ \ge \mu^{ - 1}}}u(t)} \right\|_{{H^1}}} \le C{e^{ - \alpha t}}{\left\| {\left( {{u_0},{u_1}} \right)} \right\|_\mathcal{H}} + \int_0^t {{e^{ - \alpha (t - s)}}} {\left\| {{P_{ \ge \mu^{ - 1}}}\left( {{{\left| u \right|}^{p - 1}}u} \right)(s)} \right\|_2}ds.
\end{align*}
Let $h(u) = {\left| u \right|^{p - 1}}u$, split $u$ into  $u = {P_{ \le {\mu ^{ - 1}}}}u + {P_{ \ge {\mu ^{ - 1}}}}u $, then
$$h(u) = h({P_{ \le {\mu ^{ - 1}}}}u) + {P_{ \ge {\mu ^{ - 1}}}}uO({\left| u \right|^{p - 1}}).$$

\noindent{\textit{Case 1.$\mbox{  }1<p<\frac{d}{d-2}$ for $d\ge3$}}\\
Bernstein's inequality and H\"older's inequality imply
\begin{align*}
{\left\| {{P_{ \ge {\mu ^{ - 1}}}}h(u)} \right\|_2} &\le {\left\| {{P_{ \ge {\mu ^{ - 1}}}}h({P_{ \le {\mu ^{ - 1}}}}u)} \right\|_2} + {\left\| {{P_{ \ge {\mu ^{ - 1}}}}\left( {{P_{ \ge {\mu ^{ - 1}}}}uO({{\left| u \right|}^{p - 1}})} \right)} \right\|_2} \\
&\le {\mu }{\left\| {\nabla h({P_{ \le {\mu ^{ - 1}}}}u)} \right\|_2} + {\left\| {{P_{ \ge {\mu ^{ - 1}}}}uO({{\left| u \right|}^{p - 1}})} \right\|_2} \\
&\le {\mu }{\left\| {\nabla {P_{ \le {\mu ^{ - 1}}}}u{{\left| {{P_{ \le {\mu ^{ - 1}}}}u} \right|}^{p - 1}}} \right\|_2} + {\left\| {{P_{ \ge {\mu ^{ - 1}}}}u} \right\|_m}{\left\| {{{\left| u \right|}^{p - 1}}} \right\|_{\frac{{{2^*}}}{{p - 1}}}},
\end{align*}
where $\frac{1}{m}+\frac{p-1}{2^*}=\frac{1}{2}$.
By Bernstein's inequality, we have
$$
{\mu }{\left\| {\nabla {P_{ \le {\mu ^{ - 1}}}}u{{\left| {{P_{ \le {\mu ^{ - 1}}}}u} \right|}^{p - 1}}} \right\|_2} \le {\mu }{\left\| {\nabla u} \right\|_2}\left\| {{P_{ \le {\mu ^{ - 1}}}}u} \right\|_\infty ^{p - 1} \le {\mu ^{-\frac{{d(p - 1)}}{{{2^*}}} + 1}}{\left\| {\nabla u} \right\|_2}\left\| {{P_{ \le {\mu ^{ - 1}}}}u} \right\|_{{2^*}}^{p - 1}.
$$
Since $1<p<\frac{d}{d-2}$, we conclude for some $\kappa>0$,
\begin{align}\label{u7}
{\left\| {{P_{ \ge {\mu ^{ - 1}}}}h({P_{ \le {\mu ^{ - 1}}}}u)} \right\|_2} \le {\mu ^{ \kappa }}{\left\| u \right\|_{{H^1}}}.
\end{align}
Applying Bernstein's inequality, we have
\begin{align*}
{\left\| {{P_{ \ge {\mu ^{ - 1}}}}u} \right\|_m} &\le {\left( {\sum\limits_{N \ge {\mu ^{ - 1}}}^\infty  {\left\| {{P_N}u} \right\|_m^2} } \right)^{1/2}} \le {\left( {\sum\limits_{N \ge {\mu ^{ - 1}}}^\infty  {{N^{2d\left( {\frac{1}{2} - \frac{1}{m}} \right) - 2}}} {N^2}\left\| {{P_N}u} \right\|_2^2} \right)^{1/2}} \\
&\le {\mu ^{ - d\left( {\frac{1}{2} - \frac{1}{m}} \right) + 1}}{\left( {\sum\limits_{N \ge {\mu ^{ - 1}}}^\infty  {{N^2}\left\| {{P_N}u} \right\|_2^2} } \right)^{1/2}}.
\end{align*}
which combined with (\ref{u7}) gives (\ref{345}) by $1<p<\frac{d}{d-2}$.
Next, we bound $\partial_tu$. From Duhamel principle, we have
\begin{align*}
{\partial _t}u(t) &= - \alpha u(t) + {e^{ - \alpha \delta }}\left[ { - \sqrt { - \Delta  + 1 - {\alpha ^2}} \sin \left( {t\sqrt { - \Delta  + 1 - {\alpha ^2}} } \right) + \alpha \cos \left( {t\sqrt { - \Delta  + 1 - {\alpha ^2}} } \right)} \right]u(t - \delta ) \\
&\mbox{  }+ {e^{ - \alpha \delta }}\cos \left( {t\sqrt { - \Delta  + 1 - {\alpha ^2}} } \right){\partial _t}u(t - \delta ) + \int_{t - \delta }^{t} {\cos \left( {(t-s)\sqrt { - \Delta  + 1 - {\alpha ^2}} } \right){e^{ - \alpha (t - s)}}} \left( {{{\left| u \right|}^{p - 1}}u} \right)(s)ds.
\end{align*}
For $\mu_1\ll\mu_0$, (\ref{345}) implies that there exist $\eta>0$ and $T_0>0$ such that
$$
\|P_{\ge\eta^{-1}}u(t)\|_{H^1}<\mu_1,
$$
for $t>T_0$. Taking $\delta$ large such that $e^{-\alpha\delta}<\mu_1$, then for $t>T_0+\delta$, it suffices to prove
$${\left\| {{P_{ \ge {\eta ^{ - 1}}}}h(u(t))} \right\|_2} \le {\eta ^{\lambda }},$$
for some $\lambda>0$.
The rest of the proof of (\ref{kj}) is the same as (\ref{345}).

\noindent{\textit{Case 2.$\mbox{  }1<p<\infty$ for $d=1$.}}\\
By Bernstein's inequality, H\"older's inequality, Sobolev embedding theorem,
$${\left\| {{P_{ \ge {\mu ^{ - 1}}}}h(u)} \right\|_2} \le \mu {\left\| {\nabla {P_{ \le {\mu ^{ - 1}}}}u} \right\|_2}{\left\| {{{\left| {{P_{ \le {\mu ^{ - 1}}}}u} \right|}^{p - 1}}} \right\|_\infty } + {\left\| {{P_{ \ge {\mu ^{ - 1}}}}u} \right\|_2}{\left\| {{{\left| u \right|}^{p - 1}}} \right\|_\infty } \le \mu.
$$
The remaining proof is the same as Case 1.

\noindent{\textit{Case 3.$\mbox{  }1<p<\infty$ for $d=2$.}}\\
The proof is also similar to Case 1, we omit it. 

\noindent{\textit{Case 4.$\mbox{   }\frac{d}{d-2}\le p<1+\frac{4}{d-2}$ for $d=3,4$.}}\\
Choosing $\delta$ sufficiently large, such that $e^{-\alpha\delta}\ll \mu_0$. Fix $t_0>\delta$, consider the interval $I\equiv[t_0-\delta,t_0]$.
By Duhamel principle, for $t\in I$, we have
$$
u(t)=S_{1,\alpha}(t-t_0+\delta)u(t_0-\delta)+S_{2,\alpha}(t-t_0+\delta)u(t_0-\delta)+\int^t_{t_0-\delta}S_{2,\alpha}(t-s)h(u(s))ds.
$$
Fix $\varepsilon$ sufficiently small, divide $I$ into subintervals $I_1, I_2,...,I_n$, such that $|I_j|\sim \varepsilon$, then $n\sim \frac{\delta}{\varepsilon}$.
Taking $p-1<R<\frac{4}{d-2}$, H\"older's inequality and Strichartz estimates give
\begin{align*}
\|u(t)\|_{L^R(I_j;L^{\frac{4d}{d-2}})}&\le |I_j|^{\frac{1}{R}-\frac{d-2}{4}}\|u(t)\|_{L^{\frac{4}{d-2}}_t(I_j;L^{\frac{4d}{d-2}}_x)}\\
 &\lesssim|I_j|^{\frac{1}{R}-\frac{d-2}{4}}E+|I_j|^{\frac{1}{R}-\frac{d-2}{4}}\|h(u(s))\|_{L^1_t(I_j;L^2)}.
\end{align*}
By H\"older's inequality, Sobolev embedding theorem,
$${\left\| {h(u(s))} \right\|_{L_t^1({I_j};{L^2})}} \le \int_{{I_j}} {{{\left\| {{{\left| u \right|}^{p - 1}}} \right\|}_{L_x^{\frac{{4d}}{{(p - 1)(d - 2)}}}}}{{\left\| u \right\|}_{L_x^\theta }}dt}  \le C(E){\left| {{I_j}} \right|^{ - \frac{{p - 1}}{R} + 1}}\left\| u \right\|_{L_t^R({I_j};{L^{4d/(d - 2)}})}^{p - 1},
$$
where $\frac{1}{\theta}+\frac{{(p - 1)(d - 2)}}{{4d}}=\frac{1}{2}$.
Thus
$$
\|u(t)\|_{L^{R}_t(I_j;L^{\frac{4d}{d-2}}_x)}\le C(E)\varepsilon^{\frac{1}{R}-\frac{d-2}{4}}\big(1+{\left| {{\varepsilon}} \right|^{ - \frac{{p - 1}}{R} + 1}}\left\| u \right\|_{L_t^R({I_j};{L^{4d/(d - 2)}})}^{p - 1}\big).
$$
Since $p\ge 2$ in Case 3, by continuity method, we have
$$\|u(t)\|_{L^{R}_t(I;L^{\frac{4d}{d-2}}_x)}\le 8C(E)\varepsilon^{\frac{1}{R}-\frac{d-2}{4}}.
$$
Summing up all the intervals, we get
\begin{align}\label{kjna}
\|u(t)\|_{L^{R}_t(I;L^{\frac{4d}{d-2}}_x)}\lesssim C(\delta,E),
\end{align}
where $C(E,\delta)$ is independent of $t$.
Again by Duhamel principle,
$${\left\| {{P_{ \ge {\mu ^{ - 1}}}}u(t_0)} \right\|_{{H^1}}} \le C(E){e^{ - \alpha \delta }} + \int_{t_0 - \delta }^{t_0} {{e^{ - \alpha (t - s)}}{{\left\| {{P_{ \ge {\mu ^{ - 1}}}}h(u(s))} \right\|}_2}} ds.$$
Hence it suffices to bound ${\left\| {{P_{ \ge {\mu ^{ - 1}}}}h(u)} \right\|_2}.$
Using similar arguments in Case 1, by Bernstein's inequality and H\"older's inequality, we have
$${\left\| {{P_{ \ge {\mu ^{ - 1}}}}h(u)} \right\|_2} \le \mu {\left\| {\nabla {P_{ \le {\mu ^{ - 1}}}}u} \right\|_2}{\left\| {{{\left| {{P_{ \le {\mu ^{ - 1}}}}u} \right|}^{p - 1}}} \right\|_\infty } + {\left\| {{P_{ \ge {\mu ^{ - 1}}}}u} \right\|_\theta }{\left\| {{{\left| u \right|}^{p - 1}}} \right\|_{\frac{{4d}}{{(p - 1)(d - 2)}}}},
$$
where $\frac{1}{\theta}+\frac{{(p - 1)(d - 2)}}{{4d}}=\frac{1}{2}$.
Applying Bernstein's inequality, we get
$${\left\| {{{\left| {{P_{ \le {\mu ^{ - 1}}}}u} \right|}^{p - 1}}} \right\|_\infty } \le {\mu ^{ - (p - 1)\frac{{d - 2}}{4}}}{\left\| u \right\|^{p-1}_{\frac{{4d}}{{d - 2}}}},$$
which combined with H\"older's inequality and (\ref{kjna}) give
\begin{align*}
&\int_{{t_0} - \delta }^{{t_0}} {{e^{ - \alpha (t - s)}}} \mu {\left\| {\nabla {P_{ \le {\mu ^{ - 1}}}}u(s)} \right\|_2}{\left\| {{{\left| {{P_{ \le {\mu ^{ - 1}}}}u(s)} \right|}^{p - 1}}} \right\|_\infty }ds \\
&\le C(E){\mu ^{1 - (p - 1)\frac{{d - 2}}{4}}}\int_{{t_0} - \delta }^{{t_0}} {{e^{ - \alpha (t - s)}}} \left\| u \right\|_{4d/(d - 2)}^{p - 1}ds \\
&\le C(E){\mu ^{1 - (p - 1)\frac{{d - 2}}{4}}}\left\| u \right\|_{L_t^R(I;{L_x^{4d/(d - 2)}})}^{p - 1} \\
&\le C(E,\delta ){\mu ^{1 - (p - 1)\frac{{d - 2}}{4}}}.
\end{align*}
This bound is acceptable since $1<p<1+\frac{4}{d-2}$.
The same arguments as Case 1 yield the desired bound for
${\left\| {{P_{ \ge {\mu ^{ - 1}}}}u} \right\|_\theta }$ by $1<p<1+\frac{4}{d-2}$.
Therefore, we finish our proof.
\end{proof}

Now, we prove the spatial localization, namely the following proposition:
\begin{Proposition}\label{22}
Let $u$ be a global solution to (\ref{1}) with $\mathcal{H}$ norm at most $E>0$. Then there exit $J=J(E)$ depending only on $E$, and functions $x_1(t),...,x_J(t):\Bbb R^+\to \Bbb R^d$, such that for any $\mu>0$ there exits $\eta=\eta(E,\mu)>0$ such that
$$\mathop {\lim \sup }\limits_{t \to \infty } {\int_{dist(x,\{ {x_1}(t),...,{x_J}(t)\} ) > {\eta ^{ - 1}}} {\left| {\nabla u} \right|} ^2} + {\left| u \right|^2} + {\left| {{\partial _t}u} \right|^2} \le \mu.$$
\end{Proposition}

Before proving Proposition \ref{22}, we first prove a weaker proposition:
\begin{Proposition}\label{33}
Let $u$ be a global solution to (\ref{1}) with $\mathcal{H}$ norm at most $E>0$. Then for $\mu_0>0$, there exits $J=J(E,\mu_0)$ and functions $\tilde{x}_1(t),...,\tilde{x}_J(t):\Bbb R^+\to \Bbb R^d$, and $\eta=\eta(E,\mu_0)>0$ such that
$$\mathop {\lim \sup }\limits_{t \to \infty } \int_{dist(x,\{ {\tilde{x}_1}(t),...,{\tilde{x}_J}(t)\} ) > {\eta ^{ - 1}}}  {\left| u \right|^2}\le \mu_0.$$
\end{Proposition}
\begin{proof}
The whole proof is divided into five parts. Fix $E>0$ and $\mu_0$,
choose parameters $\mu_0\gg\mu_1\gg\mu_2\gg\mu_3\gg\mu_4>0$.\\

{\textit{Step One. Selecting a ``good" time sequence}}
For any $t_0>T_0$, consider the time interval $[t_0-\mu_1^{-1}, t_0+\mu_1^{-1}]$. Since
$$\mathop {\lim }\limits_{{t_0} \to \infty } \int_{{t_0} - \mu _1^{ - 1}}^{{t_0} + \mu _1^{ - 1}} {\left\| {{\partial _t}u(s)} \right\|_2^2} ds = 0,$$
there exists $T_1$ sufficiently large such that for $t_0>T_1$,
$$\int_{{t_0} - \mu _1^{ - 1}}^{{t_0} + \mu _1^{ - 1}} {\left\| {{\partial _t}u(s)} \right\|_2^2}ds\le \mu_2^2.$$
Thus there exits good time $t_*\in [t_0-\mu_1^{-1}, t_0+\mu_1^{-1}]$, such that
\begin{align}\label{pk}
\|\partial_tu(t_*)\|_2\le \mu_2^2.
\end{align}

{\textit{Step Two. $L^{\infty}_x$ spatial localization at fixed time}.}
From Lemma 3.1, for any $\mu_2>0$ there exists $c(\mu_2)>0$, such that for $T>T_0$,
\begin{align}\label{s1}
\|u_{>c(\mu_2)^{-1}}\|_{H^1}\le \mu_2^2.
\end{align}
As step one, we fix time $t>T_1$. Now we claim there exist $J(E,\mu_2,\mu_3)$ and $x_1(t),...,x_J(t):\Bbb R^+ \to \Bbb R^d$, such that
\begin{align}\label{45}
|u_{<c(\mu_2)^{-1}}(t,x)|<\mu_3, \mbox{  }{\rm{whenever}}\mbox{  }dist(x,\{x_1(t),...,x_J(t)\}) \ge 2\mu_3^{-1}.
\end{align}
Indeed, let $x_1(t),...,x_J(t)$ be a maximal $2\mu_3^{-1}$-separated set of points in $R^d$ such that
\begin{align*}
|u_{<c(\mu_2)^{-1}}(t,x_j(t))|\ge \mu_3 \mbox{  }\mbox{  }{\rm{for}}\mbox{  }\mbox{  }{\rm{all}} \mbox{  }\mbox{  } 1\le j\le J(t).
\end{align*}
It is easy to verify
$$
|u_{<c(\mu_2)^{-1}}(t,x_j(t))|\lesssim c(\mu_2)^{d/2}\int_{|x-x_j(t)|\le \mu_3^{-1}}|u|^2dx+\mu_3^d\|u\|_2.
$$
Then we have
$$
\mu_3\le |u_{<c(\mu_2)^{-1}}(t,x_j(t))|\lesssim c(\mu_2)^{d/2}\int_{|x-x_j(t)|\le \mu_3^{-1}}|u|^2dx.
$$
Since $x_j(t)$ are $2\mu_3^{-1}$-separated, thus $J$ is finite depending on $\mu_2,\mu_3$. By the maximal property of the set $\{x_1,...,x_J\}$,
we conclude
$$
|u_{<c(\mu_2)^{-1}}(t,x_j(t))|<\mu_3, \mbox{  }\mbox{  } {\rm{whenever}} \mbox{  }\mbox{  }dist(x,\{x_1,...,x_J\})\ge 2\mu^{-1}_3.
$$

{\textit{Step Three. $L^{\infty}_x$ spatial localization on an interval centered at  good time}.}
For $t>T_1$, consider good time $t_*$ in $[t-\mu_1^{-1},t+\mu_1^{-1}]$. Then $[t-\mu_1^{-1},t+\mu_1^{-1}]\subset[t_*-4\mu_1^{-1},t_*+4\mu_1^{-1}]\equiv I$. Define the distance function $D(x)=dist(x,\{x_1(t_*),x_2(t_*),...,x_J(t_*)\})$. Let $\chi:\Bbb R^d\to R^+$ be a smooth cutoff function which equals 1 for $D(x)\le 2\mu_3^{-1}$, vanishes for $D(x)\ge 3\mu_3^{-1}$, and $\nabla^k\chi=O_k(\mu_3^k)$ for $k\ge0$. Then we have

\noindent {\textit{Claim 1.}} $\|S_{1,\alpha}[(1-\chi)u(t_*)]\|_{L^{Q}_t(I;L^{2p}_x)}+\|S_{2,\alpha}[(1-\chi)u(t_*)]\|_{L^{\theta*}_t(I;L^{2\theta^*}_x)}\lesssim_{\mu_1} \mu_2^2.$\\
By Strichartz estimates,
$$\|S_{2,\alpha}[(1-\chi)u(t_*)]\|_{L^{Q}_t(I:L^{2p}_x)} \le Ce^{C\mu_1^{-1}}\|\partial_tu(t_*)\|_2,$$
which combined with (\ref{pk}) yields the desired bounds for $S_{2,\alpha}$. Since high frequency is small by (\ref{s1}), it suffices to prove
$${\left\| {{S_{1,\alpha}}[(1 - \chi ){P_{ \le C(\mu _2)^{-1}}}u({t_*})]} \right\|_{L_t^{Q}(I;L_x^{2{p}})}}\lesssim_{\mu_1}{\mu _2}^2.$$
From the rapid decay of the convolution kernel of $P_{<c(\mu_2)^{-1}}$ and the support of $1-\chi$, we see that $(1-\chi)P_{<c(\mu_2)^{-1}}(1_{D<\mu_3^{-1}}P_{<c(\mu_2)^{-1}}u)$ can be absorbed by $\mu_2^2$, it suffices to prove
$${\left\| {{S_{1,\alpha }}\left[ {(1 - \chi ){P_{ < C{{({\mu _2})}^{ - 1}}}}\left( {{1_{D > \mu _3^{ - 1}}}{P_{ < C{{({\mu _2})}^{ - 1}}}}u({t_*})} \right)} \right]} \right\|_{L_t^{{Q}}(I;L_x^{2p})}} \le \mu _2^2.
$$
Indeed, stationary phase shows that the operator ${{S_{1,\alpha }}(1 - \chi ){P_{ < C{{({\mu _2})}^{ - 1}}}}}$ have an operator norm of $C_1(\mu_2)$ on $L^{2\theta^*}$, then thanks to (\ref{45}), for some $\delta>0$, we have
\begin{align*}
&{\left\| {{S_{1,\alpha }}\left[ {(1 - \chi ){P_{ < C{{({\mu _2})}^{ - 1}}}}\left( {{1_{D > \mu _3^{ - 1}}}{P_{ < C{{({\mu _2})}^{ - 1}}}}u({t_*})} \right)} \right]} \right\|_{L_t^{{Q}}(I;L_x^{2{p}})}} \\
&\le {C_1}({\mu _2}){\left\| {{1_{D > \mu _3^{ - 1}}}{P_{ < C{{({\mu _2})}^{ - 1}}}}u({t_*})} \right\|_{L_t^{{Q}}(I;L_x^{2{p}})}} \\
&\le {C_1}({\mu _2})\left\| {{1_{D > \mu _3^{ - 1}}}{P_{ < C{{({\mu _2})}^{ - 1}}}}u({t_*})} \right\|_{L_t^{{Q}}(I;L_x^\infty )}^\delta \left\| {{1_{D > \mu _3^{ - 1}}}{P_{ < C{{({\mu _2})}^{ - 1}}}}u({t_*})} \right\|_{L_t^{{Q}}(I;L_x^{2})}^{1 - \delta } \\
&\le {C_1}({\mu _2},{\mu _1})\mu _3^\delta  \lesssim {\mu_2 ^2}.
\end{align*}

\noindent {\textit{Claim 2.}} $\|1_{D>\mu_4^{-2}}u\|_{L^{Q}_t(I:L^{2p}_x)}\lesssim_{\mu_1}{\mu _2}.$\\
This claim can be proved by perturbation theorem and Strichartz estimates. Indeed, let $v$ be a solution to (\ref{1}) on $I$ with initial data $v(t_*)=\chi u(t_*)$, $\partial_tv(t_*)=\partial_t u(t_*)$. Then by perturbation theorem, Claim 1, (\ref{pk}), we have
$$
\|u-v\|_{L_t^{{Q}}(I;L_x^{2p})}\lesssim_{\mu_1} \mu_2^2.
$$
It suffices to prove
\begin{align}\label{dfr}
\|1_{D>\mu_4^{-2}}v\|_{L_t^{{Q}}(I;L_x^{2p})}\lesssim_{\mu_1} \mu_2^2.
\end{align}
Choosing another weight function $W:\Bbb R^d\to \Bbb R^+$ comparable to $1+\mu_4D$ which obeys the bounds $\nabla W, \nabla^2W=O(\mu_4)$. Since $W\chi=O(1)$, we have
$$
\|Wv(t_0)\|_2\lesssim1.
$$
Since $v$ solves (1), we have
$$
{\partial _{tt}}(Wv) + 2\alpha {\partial _t}(Wv) - \Delta (Wv) + Wv = W{\left| v \right|^{p - 1}}v + O({\mu _4}\left| v \right|) + O({\mu _4}\left| {\nabla v} \right|).
$$
Strichartz estimates imply
$${\left\| {Wv} \right\|_{L_t^{{\theta ^*}}(I;L_x^{2{\theta ^*}})}} + {\left\| {\big(Wv,\partial_t(Wv)\big)} \right\|_{L_t^\infty (I;\mathcal{H})}} \lesssim {\left\| {(Wv(t'),{\partial _t}Wv(t'))} \right\|_\mathcal{H}} + {\left\| {W{{\left| u \right|}^p}} \right\|_{L_t^1(I;L_x^2)}} + {\mu _4},
$$
for any subinterval $I'$ of $I$ and any $t'\in I'$. Denote the left side by $X(I')$, H\"older's inequality and Sobolev embedding theorem reveal that $$
{\left\| {W{{\left| u \right|}^p}} \right\|_{L_t^1(I';L_x^2)}} \le C{\left| {I'} \right|^{\frac{4}{{d + 2}}}}X(I').
$$
Chopping $I$ up to sufficiently small intervals, we have
$$X(I')\le C(\mu_1);$$
particularly, we conclude
$$
\|1_{D>\mu^{-2}_4}v\|_{L_t^{{\theta ^*}}(I;L_x^{2\theta^*})}\le C(\mu_1)\mu_4,
$$
which yields (\ref{dfr}) by interpolation inequality thus finishing the proof of Claim 2.

\noindent {\textit{Claim 3.}} $\|1_{D>\mu_4^{-3}}\int_IS_{2,\alpha}(t-s)\big(|u|^{p-1}u\big)(s)ds\|_{L^2(\Bbb R^d}\lesssim_{\mu_1}\mu _2.$ \\
From finite speed of propagation, $S_{2,\alpha}(t-s)\big(|u|^{p-1}u1_{D\le \mu_4^{-2}}\big)$ is supported in $D\le \mu_4^{-2}+4\mu_1^{-1}$.
Therefore $1_{D>\mu_4^{-3}}\int_IS_{2,\alpha}(t-s)\big(|u|^{p-1}u(s)1_{D\le \mu_4^{-2}}\big)ds=0$. Thus it suffices to prove
\begin{align}\label{ko}
{\left\| {\int_I {{S_{2,\alpha}}} (t - s)|u{|^{p - 1}}u(s){1_{D \ge \mu _4^{ - 2}}}ds} \right\|_{{L^2}}}{ \mathbin{\lower.3ex\hbox{$\buildrel<\over
{\smash{\scriptstyle\sim}\vphantom{_x}}$}} _{{\mu _1}}}{\mu _2}.
\end{align}
By Strichartz estimates, the left is bounded by
$$\||u|^p1_{D\ge \mu_4^{-2}}\|_{L^1_tL^2_x}.$$
Then (\ref{ko}) follows from H\"older's inequality and Claim 2.

{\textit{Step Four. $L^2$ localization of good times}.}
In this step, we prove the $L^2$ localization of $u(t_*)$, namely for $T_1$ sufficiently large, $t>T_1$,
\begin{align}\label{ol}
\|1_{D>\mu_4^{-3}}u(t_*)\|_2=O_{L^2}(\mu_1).
\end{align}
The proof is based on the decay of linear part and Claim 3 in step three. Indeed, from Duhamel principle and Claim 3,
$$1_{D>\mu_4^{-3}}u(t_*)=1_{D>\mu_4^{-3}}S_{1,\alpha}(\mu_1^{-1})u(t_*-\mu_1^{-1})+1_{D>\mu_4^{-3}}S_{2,\alpha}(\mu-1^{-1})u_t(t_*-\mu_1^{-1})+O_{L^2}(\mu_2).
$$
Then since $S_{1,\alpha}$ and $S_{2,\alpha}$ have a exponential decay, we obtain
\begin{align*}
\left\| {{1_{D > \mu _4^{ - 3}}}u({t_*})} \right\|_2^2 &= \left\langle {{1_{D > \mu _4^{ - 3}}}u({t_*}),{S_{1,\alpha}}(\mu _1^{ - 1})u({t_*} - \mu _1^{ - 1})} \right\rangle  + \left\langle {{1_{D > \mu _4^{ - 3}}}u({t_*}),{S_{2,\alpha}}(\mu _1^{ - 1}){\partial _t}u({t_*} - \mu _1^{ - 1})} \right\rangle  + O({\mu _2}) \\
&\le {e^{ - \mu _1^{ - 1}\alpha }}\left\| {u({t_*})} \right\|_2^2 + O({\mu _2}) \lesssim {\mu _1}.
\end{align*}
Thus (\ref{ol}) follows.

{\textit{Step Five. $L^2$ localization of all time}.}
First from Duhamel principle and similar arguments as step four, it is easy to verify,
$$\|1_{D\ge \mu_4^{-3}}u(t)\|_2\le \mu_1,$$
for $t\in(t_*, t_*+4\mu_1^{-1})$.
Indeed, from Duhamel principle, finite speed of propagation, Claim 2 and Claim 3, we have
\begin{align*}
{\left\| {{1_{D > \mu _4^{ - 4}}}u(t)} \right\|_2} &\le {\left\| {{1_{D > \mu _4^{ - 4}}}{S_{1,\alpha }}(t - {t_*})u({t_*})} \right\|_2} + {\left\| {{1_{D > \mu _4^{ - 4}}}{S_{2,\alpha }}(t - {t_*})u({t_*})} \right\|_2} \\
&+ {\left\| {{1_{D > \mu _4^{ - 4}}}\int_{{t_*}}^t {{e^{ - \alpha (t - s)}}{S_{2,\alpha }}(t - s)\left( {{{\left| u \right|}^{p - 1}}u} \right)(s)} } \right\|_2} \\
&\lesssim {\left\| {{1_{D > \mu _4^{ - 3}}}u({t_*})} \right\|_2} + {\left\| {{1_{D > \mu _4^{ - 3}}}u({t_*})} \right\|_2} + {\mu _2} \lesssim {\mu _1}.
\end{align*}
Splitting the whole interval $[T_1,\infty)$ into subintervals with length $2\mu_1^{-1}$, denote these subintervals as $I_1, I_2, I_3,...$.  Denote $t_*\in I_j$ by $t^j_*$. It is obvious that $I_{j+1}$ is covered by $(t^j_*,t^j_* + 4\mu _1^{ - 1})$, for $j=1,2,...$. Now let's define $\tilde{x}_j(t)$ for each $t$ by the following rule:
For $t\in I_{j+1}$, take $\tilde{x}_j(t)=x_j(t^j_*)$.
It is direct to see $\tilde{x}_j(t)$ defined above satisfies Proposition \ref{33} for all $t>T_1+2\mu_1^{-1}$.
\end{proof}

\begin{Proposition}\label{44}
Let $u$ be a global solution to (\ref{1}) with $\mathcal{H}$ norm at most $E>0$. Then for $\mu_0>0$, there exits $J=J(E,\mu_0)$ and functions $\tilde{x}_1(t),...,\tilde{x}_J(t):\Bbb R^+\to \Bbb R^d$, and $\eta=\eta(E,\mu_0)>0$ such that
$$\mathop {\lim \sup }\limits_{t \to \infty } {\int_{dist(x,\{ {\tilde{x}_1}(t),...,{\tilde{x}_J}(t)\} ) > {\eta ^{ - 1}}} {\left| {\nabla u} \right|} ^2} + {\left| u \right|^2} + {\left| {{\partial _t}u} \right|^2} \le \mu_0.$$
\end{Proposition}
\begin{proof}
Choose $\mu_4\ll\mu_3\ll\mu_2\ll\mu_1\ll\mu_0$ as Proposition \ref{33}. Suppose that $t\in I_{j+1}$, then $0<t-t^j_*<4\mu_1^{-1}$.
Define $D(t)={dist(x,\{ {\tilde{x}_1}(t),...,{\tilde{x}_J}(t)\})}$. Then the proof of Proposition \ref{33} implies that for all $t\in I_{j+1}$,
\begin{align}\label{xcv}
1_{D(t)}=1_{D(t_*^j)}.
\end{align}
 Let $\chi_1$ be a cutoff function supported in $D(t^j_*)>\mu_4^{-4}$, which equals one in $D>\mu_4^{-5}$, with bound $|\nabla\chi_1|\lesssim \mu_4.$ Therefore we have
$$ \int_{D(t_*^j)> \mu _4^{ - 5}} {{{\left| {\nabla u(t)} \right|}^2} dx\le } {\left\| {u(t){\chi _1}} \right\|_{{H^1}}} + {\mu _4}.$$
Duhamel principle and finite speed of propagation give
\begin{align*}
u(t)\chi_1 &= {S_{1,\alpha }}(\mu _1^{ - 1}){1_{D(t_*^j) > \mu _4^{ - 3}}}u(t - \mu _1^{ - 1}) + {S_{2,\alpha }}(\mu _1^{ - 1}){1_{D(t_*^j) > \mu _4^{ - 3}}}{\partial _t}u(t - \mu _1^{ - 1}) \\
&+ \int_{t - \mu _1^{ - 1}}^t {{S_{2,\alpha }}(t - s) {{{\left| u \right|}^{p - 1}}u} } {1_{D(t_*^j) > \mu _4^{ - 3}}}(s)ds.
\end{align*}
By Strichartz estimates and exponential decay of ${S_{1,\alpha }},S_{2,\alpha}$, we get
\begin{align*}
{\left\| {u(t){\chi _1}} \right\|_{{H^1}}} \le C(E){e^{ - \alpha \mu _1^{ - 1}}}
+ {\left\| {{{\left| u \right|}^{p - 1}}u{1_{D(t_*^j) > \mu _4^{ - 3}}}} \right\|_{L_t^{{1}}((t - \mu _1^{ - 1},t);L_x^{2})}}.
\end{align*}
Then Claim 2, $(t-\mu^{-1}_{1},t)\subset (t^j_*-4\mu_1^{-1},t^j_*+4\mu_1^{-1})$, and (\ref{xcv}) imply,
\begin{align}\label{bounds}
\int_{D(t) > \mu _4^{ - 5}} {{{\left| {\nabla u(t)} \right|}^2} dx\le } {\mu _1},
\end{align}
which gives us the desired bound for $\nabla u(t)$.\\
Next, we prove the desired bound for $\partial_tu$.
By Duhamel principle, and finite speed of propagation, we obtain
\begin{align*}
{1_{D(t_*^j) > \mu _4^{ - 4}}}u(t) &= {S_{1,\alpha }}(t - t_*^j){1_{D(t_*^j) > \mu _4^{ - 3}}}u(t_*^j) + {S_{2,\alpha }}(t - t_*^j){1_{D(t_*^j) > \mu _4^{ - 3}}}{\partial _t}u(t_*^j)\\
&+ \int_{t_*^j}^t {{e^{ - \alpha (t - s)}}}S_{2,\alpha}(t-s) {1_{D(t_*^j) > \mu _4^{ - 3}}}{\left| {u(s)} \right|^{p - 1}}u(s)ds.
\end{align*}
Direct calculations yield
\begin{align*}
{\left\| {{\partial _t}\left[ {{1_{D(t_*^j) > \mu _4^{ - 4}}}u(t)} \right]} \right\|_2} \le&{\left\| {{1_{D(t_*^j) > \mu _4^{ - 3}}}u(t_*^j)} \right\|_{{H^1}}} + {\left\| {{1_{D(t_*^j) > \mu _4^{ - 3}}}{\partial _t}u(t_*^j)} \right\|_2} \\
&+ {\left\| {{1_{D(t_*^j) > \mu _4^{ - 3}}}{{\left| u \right|}^{p - 1}}u} \right\|_{L_t^{{1}}(I;L_x^{2})}},
\end{align*}
where $I=[t^j_*-\mu_1^{-1},t^j_*+\mu_1^{-1}]$. Then Claim 2, (\ref{pk})  and (\ref{bounds}) imply
$${\left\| {{\partial _t}\left[ {{1_{D(t^j_*) > \mu _4^{ - 4}}}u(t)} \right]} \right\|_2} \lesssim {\mu _1}.$$
Then the desire bound follows from (\ref{xcv}).
\end{proof}

From the proof of Proposition 10.1 in T. Tao \cite{TT}, Proposition \ref{22} is a corollary of Proposition \ref{44} and the following lemma.
\begin{Lemma}\label{gy}
Let $u$ be a global solution with $\mathcal{H}$ norm at most $E$. Suppose that we have the energy concentration bound
$$
\int_{|x-x_0|<R}|u(t_0,x)|^2+|\nabla u(t_0,x)|^2+|\partial_t u(t_0,x)|^2dx\ge \eta_1^2
$$
for some $x_0\in \Bbb R^d$, $t_0\in \Bbb R^+$, $R>0$, and sufficiently small $\eta_1>0$. Then, if $t_0$ is sufficiently large depending on $u,E,x_0,R.\eta_1$, we have the improved energy concentration
$$
\int_{|x-x_0|<R'}|u(t_0,x)|^2+|\nabla u(t_0,x)|^2+|\partial_t u(t_0,x)|^2dx\ge \beta(E),
$$
for some $\beta(E)>0$ independent of $\eta_1$ and some $R'$ depending on $E,R,\eta_1$.
\end{Lemma}

The proof of Lemma \ref{gy} can be reduced to the following lemma.
\begin{Lemma}
Given $E>0$, $\eta_1>0$ sufficiently small, there exists $\beta>0$ with the following property:
Suppose that we have the energy concentration bound
$$
\eta_1^2\le \int_{|x-x_0|<R}{\left| {\nabla u}(t_0,x) \right|} ^2 + {\left| u(t_0,x) \right|^2} + \left| {{\partial _t}u(t_0,x)} \right|^2\le \beta
$$
for some $x_0\in \Bbb R^d$, $t_0\in \Bbb R^+$, $R>0$, and some global solution $u$ with $\mathcal{H}$ norm at most $E$. Then, if $t_0$ is sufficiently large depending on $E,x_0,R, \eta_1$, we have
$$
\int_{|x-x_0|<R'}{\left| {\nabla u(t_0,x)} \right|} ^2 + {\left| u (t_0,x)\right|^2} + \left| {{\partial _t}u(t_0,x)} \right|^2\ge \int_{|x-x_0|<R}{\left| {\nabla u(t_0,x)} \right|} ^2 + {\left| u (t_0,x)\right|^2} + \left| {{\partial _t}u(t_0,x)} \right|^2+\eta_4^2,
$$
for some $\eta_4(E,\eta_1)>0$ and $R'(E,R,\eta_1,\eta_4)$.
\end{Lemma}
\begin{proof}
For simplicity, define $e={\left| {\nabla u(t_0,x)} \right|} ^2 + {\left| u (t_0,x)\right|^2} + \left| {{\partial _t}u(t_0,x)} \right|^2$.
Fix $E>0$, let $\beta>0$ be a sufficiently small quantity to be determined. Choose parameters
$\eta_1\gg\eta_2\gg\eta_3\gg\eta_4>0$. Let $R_0>max(32R, \frac{32}{\eta_3})$. Suppose by contradiction that there exists $R'>R_0$ such that our claim fails, then
$$\int_{R<|x-x_0|<R'}e(t_0,x)dx\lesssim \eta_4^2,$$
especially, we have
$$
\int_{|x|<R'}e(t_0,x)\lesssim \beta.
$$
Choose $\beta<\epsilon$, where $\epsilon$ is the constant in Proposition \ref{local}. Denote the solution of (1.1) with initial data  $1_{|x|<R'}u(t_0,x)$ at $t_0$. Proposition \ref{local} implies
\begin{align}\label{fg}
\|\tilde{u}\|_{\mathcal{H}}\le e^{-\gamma (t-t_0)}\|\tilde{u}(t_0)\|_{\mathcal{H}}\lesssim \beta .
\end{align}
Finite speed of propagation implies
$$u(x,t)=\tilde{u}(x,t) \mbox{  }{\rm{in}}\mbox{   } \{(x,t):|x|<R'-|t-t_0|\}.
$$
Consider a time interval $I=[t_0,t_0+\eta_3^{-1}]$. Then for $t\in I$, we have
$$\int_{|x|<R'/2}e(t,x)dx=\int_{|x|<R'/2}\tilde{e}(t,x)dx.$$
Combining with (\ref{fg}), we have verified
$$
\int_{|x|<R'/2}e(t+\eta_3^{-1},x)dx\lesssim e^{-\eta_3^{-1}}\beta\lesssim \eta^3.
$$
If we have obtained
\begin{align}\label{cv}
\inf_{t\in I}\int_{|x|<R'/2}e(t,x)\ge \eta_1^2,
\end{align}
then contradiction follows. Hence, it suffices to prove (\ref{cv}). By Lemma 3.1, there exists $\mu>0$, $T_0>0$, such that for any $t>T_0$,
$$
\|P_{>\mu^{-1}}u\|_{H^1}+\|P>{\mu^{-1}}\partial_tu\|_{L^2}\le \eta_1^6.
$$
Hence it suffices to prove
$$
\mathop {\inf }\limits_{t \in I} {\int_{\left| x \right| < R'/2} {\left| {{P_{ \le {\mu ^{ - 1}}}}u} \right|} ^2} + {\left| {\nabla {P_{ \le {\mu ^{ - 1}}}}u} \right|^2} + {\left| {{\partial _t}{P_{ \le {\mu ^{ - 1}}}}u} \right|^2}dx \gtrsim \eta _1^2.
$$
Let $\psi(x)$ be a smooth cutoff function which equals 1 in $\{|x|<R'/4\}$, vanishes when $|x|>R'/2$, with bound $|\nabla \psi(x)|=O(R'^{-1})$, then
\begin{align*}
&\frac{d}{{dt}}\int_{{\Bbb R^d}} {\psi (x)} \left[ {{{\left| {{P_{ \le {\mu ^{ - 1}}}}u} \right|}^2} + {{\left| {\nabla {P_{ \le {\mu ^{ - 1}}}}u} \right|}^2} + {{\left| {{P_{ \le {\mu ^{ - 1}}}}{u_t}} \right|}^2}} \right]dx \\
&= 2\int_{{\Bbb R^d}} {\psi (x)} \left( {{P_{ \le {\mu ^{ - 1}}}}{u_t}} \right){P_{ \le {\mu ^{ - 1}}}}h(u) - 4\alpha \int_{{\Bbb R^d}} {\psi (x)} {\left| {{P_{ \le {\mu ^{ - 1}}}}{u_t}} \right|^2}dx - 2\int_{{\Bbb R^d}} {\nabla \psi (x)} \left( {\nabla {P_{ \le {\mu ^{ - 1}}}}u} \right)\left( {{P_{ \le {\mu ^{ - 1}}}}{u_t}} \right)dx.
\end{align*}
Define
$${e_\mu }(t) = \left| {{P_{ \le {\mu ^{ - 1}}}}u} \right| + {\left| {\nabla {P_{ \le {\mu ^{ - 1}}}}u} \right|^2} + {\left| {{\partial _t}{P_{ \le {\mu ^{ - 1}}}}u} \right|^2}.$$
H\"older's inequality yield,
\begin{align*}
&\left| {\int_{{\Bbb R^d}} {\psi (x)} {e_\mu }(t)dx - \int_{{\Bbb R^d}} {\psi (x)} {e_\mu }({t_0})dx} \right| \\
&\le \int_{{t_0}}^{{t_0} + 1/{\eta _3}} {\left| {\frac{d}{{dt}}\int_{{\Bbb R^d}} {\psi (x)} {e_\mu }dx} \right|dt}  \\
&\le C\int_{{t_0}}^{{t_0} + 1/{\eta _3}} {{{\left\| {{u_t}} \right\|}_2}{{\left\| {{P_{ \le {\mu ^{ - 1}}}}h(u)} \right\|}_2} + \left\| {{u_t}} \right\|_2^2 + {{\left\| {{u_t}} \right\|}_2}{{\left\| {\nabla u} \right\|}_2}dtdt}.
\end{align*}
For $d\ge3$, $1<p\le \frac{d}{d-2}$, Sobolev embedding theorem implies ${{{\left\| {{P_{ \le {\mu ^{ - 1}}}}h(u)} \right\|}_2}}\le C(E)$, thus
\begin{align}\label{guf}
\left| {\int_{{\Bbb R^d}} {\psi (x)} {e_\mu }(t)dx - \int_{{\Bbb R^d}} {\psi (x)} {e_\mu }({t_0})dx} \right| \le C(E,\mu)\int_{{t_0}}^{{t_0} + 1/{\eta _3}} {\left( {\left\| {{u_t}} \right\|_2^2 + {{\left\| {{u_t}} \right\|}_2}} \right)} dt.
\end{align}
For $d\ge3$, $\frac{d}{d-2}<p<1+\frac{4}{d-2}$, by Bernstein's inequality,
$${\left\| {{P_{ \le {\mu ^{ - 1}}}}h(u)} \right\|_2} \le {\mu ^{ - d\left( {\frac{p}{{{2^*}}} - \frac{1}{2}} \right)}}{\left\| {{P_{ \le {\mu ^{ - 1}}}}h(u)} \right\|_{\frac{{{2^*}}}{p}}} \le {\mu ^{ - d\left( {\frac{p}{{{2^*}}} - \frac{1}{2}} \right)}}C(E),$$
which yields (\ref{guf}) again.
For $d=1,2$, (\ref{guf}) can be obtained directly by Sobolev embedding theorem.
Since ${\int_0^\infty  {\left\| {{u_t}} \right\|_2^2dt}  < \infty }$,
choose $t_0$ sufficiently large such that
$$\int_{{t_0}}^{{t_0} + 1/{\eta _3}} {\left\| {{u_t}} \right\|_2^2} dt \le \sqrt {{\eta _3}} {\rm{ }}{\mu ^{d\left( {\frac{p}{{{2^*}}} - \frac{1}{2}} \right)}}\eta _1^6,
$$
then
$$\int_{|x| < R'/2} {e_{\mu}(t)} dx \ge \int_{{\Bbb R^d}} {\psi (x)e_{\mu}(t)} dx \ge \int_{{\Bbb R^d}} {\psi (x)e_{\mu}({t_0})} dx -\sqrt {{\eta _3}} \eta _1^3\gtrsim \eta _1^2,$$
thus proving (\ref{cv}), from which our lemma follows.
\end{proof}

\section{Concentration compact attractor}
In this section, we first derive the global attractor, then we prove Theorem 1.1 immediately.
\subsection{Concentration-compactness attractor}
We recall the following criterion for compact attractors proved by Proposition B.2 in Tao \cite{TT}.
\begin{Proposition}\label{kkkk}
Let $\mathcal{U}$ be a collection of trajectories $u:\Bbb R^+\to \mathcal{H}$. If $\mathcal{U}$ is bounded in $\mathcal{H}$, and for any $\mu_0>0$ there exists $\mu_1>0$ such that
\begin{align*}
&\mathop {\lim \sup }\limits_{t \to \infty } {\left\| {{P_{ > 1/{\mu _1}}}u(t)} \right\|_\mathcal{H}} \le {\mu _0}, \\
&\mathop {\lim \sup }\limits_{t \to \infty } {\int_{|x| > 1/{\mu _1}} {\left| {u(x,t)} \right|} ^2} + {\left| {\nabla u(x,t)} \right|^2} + {\left| {{\partial _t}u(x,t)} \right|^2}dx \le {\mu _0}.
\end{align*}
Then there exists a compact set $K\subset\mathcal{H}$ such that $\mathop {\lim }\limits_{t \to \infty }dist_{\mathcal{H}}(u(t),K)=0$.
\end{Proposition}

\begin{Proposition}\label{pph}
Let $\mathcal{U}$ be a collection of trajectories $u:\Bbb R^+\to \mathcal{H}$, and let $J\ge1$. If $\mathcal{U}$ is bounded in $\mathcal{H}$, and for any $\mu_0>0$ there exists $\mu_1>0$ such that for every $u\in \mathcal{U}$ we have $x_1,...,x_J:\Bbb R^+\to\Bbb R^d$ for which
\begin{align*}
&\mathop {\lim \sup }\limits_{t \to \infty } {\left\| {{P_{ > 1/{\mu _1}}}u(t)} \right\|_\mathcal{H}} \le {\mu _0}, \\
&\mathop {\lim \sup }\limits_{t \to \infty } {\int_{dist\left( {x,\left\{ {{x_1}(t),{x_2}(t),...,{x_J}(t)} \right\}} \right) > 1/{\mu _1}} {\left| {u(x,t)} \right|} ^2} + {\left| {\nabla u(x,t)} \right|^2} + {\left| {{\partial _t}u(x,t)} \right|^2}dx \le {\mu _0}.
\end{align*}
Then there exists a G-precompact set $K\subset\mathcal{H}$ with $J$ components such that $\mathop {\lim }\limits_{t \to \infty }dist_{\mathcal{H}}(u(t),K)=0$.
\end{Proposition}
\begin{proof} Although the proof is almost the same as proposition B.3 of T. Tao \cite{TT}, for reader's convenience, we give a sketch here.
We use the partition of unity
$$
 1 = \sum\limits_{j = 1}^J {{\psi _{j,t}}(x)},
$$
where
$$
 {\psi _{j,t}}(x) \equiv \frac{{{{\left\langle {x - {x_j}(t)} \right\rangle }^{ - 1}}}}{{\sum\limits_{l = 1}^J {{{\left\langle {x - {x_l}(t)} \right\rangle }^{ - 1}}} }}.
$$
Split $(u(x,t),\partial_tu(x,t))$ as
\begin{align}\label{oi}
u(t) = \sum\limits_{j = 1}^J {{\tau _{{x_j}(t)}}} {w_j}(t),{\rm{   }}{\partial _t}u(t) = \sum\limits_{j = 1}^J {{\tau _{{x_j}(t)}}} {v_j}(t),
\end{align}
where
$${w_j}(t) = {\tau _{ - {x_j}(t)}}{\psi _{j,t}}(x)u(t),\mbox{  }{v_j}(t) = {\tau _{ - {x_j}(t)}}{\psi _{j,t}}(x){\partial _t}u(t).
$$
The localization of $u$ and $\partial_tu$ implies for any $\mu_0>0$, there exists $\eta>0$ such that
\begin{align*}
\mathop {\lim \sup }\limits_{t \to \infty } {\left\| {{P_{ > \eta }}{w_j}} \right\|_{{H^1}}} + {\left\| {{P_{ > \eta }}{v_j}} \right\|_{{L^2}}} \le {\mu _0},\mbox{  }\mbox{  }\mathop {\lim \sup }\limits_{t \to \infty } \int_{\left| x \right| > \eta } {{{\left| w_j \right|}^2} + {{\left| {\nabla w_j} \right|}^2} + {{\left| {v_j} \right|}^2}}  \le {\mu _0}.
\end{align*}
From Proposition \ref{kkkk}, there exist a compact set $K_1\subset H^1$ and a compact set $K_2\subset L^2$, such that
$$
\mathop {\lim }\limits_{t \to \infty } dist({w_j}(t),{K_1}) = 0,\mbox{  }\mbox{  }\mathop {\lim }\limits_{t \to \infty } dist({v_j}(t),{K_2}) = 0,
$$
for all $j=1,2,...,J$. Combining with (\ref{oi}), we obtain
$$
dist_\mathcal{H}(u,J(GK))=0,
$$
where $K=K_1\times K_2$.
\end{proof}

As a corollary of Proposition \ref{pph}, Proposition \ref{22}, Lemma \ref{13}, we have
\begin{Corollary}\label{hji}
There exists a compact set $K\subset \mathcal{H}$ and $0\le J<\infty$, such that
$$\mathop {\lim }\limits_{t \to \infty } dist_{\mathcal{H}}(u(t),J(GK))=0.$$
\end{Corollary}

\section{Proof of theorem 1.1}
{\textit{Step one.}}
Combining Corollary \ref{hji} with Lemma B.7 in Tao \cite{TT}, we have for any $t_n\to \infty$, up to a subsequence there exits $J_1,J_2,...,J_M$ and $w_m\in J_m(GK)$ such that
\begin{align*}
 u({t_n}) &= \sum\limits_{m = 1}^M {{\tau _{{x_{m,n}}}}} {w_m} + {o_{H^1}}(1) \\
 {\partial _t}u({t_n}) &= \sum\limits_{m = 1}^M {{\tau _{{x_{m,n}}}}} {v_m} + {o_{L^2}}(1),\\
\end{align*}
where $x_{m,n}\in \Bbb R^d$ and they satisfies
$\mathop {\lim }\limits_{n \to \infty } \left| {{x_{m,n}} - {x_{k,n}}} \right| = \infty$, for $k\neq m$.

{\textit{Step two.}}
By linear energy decoupling property, we have $\mathop {\sup }\limits_m {\left\| {\left( {{w_m},{v_m}} \right)} \right\|_\mathcal{H}} < C$, by the local theory, there exists $T>0$ such that the solution $W_j$ to (\ref{1}) with initial data $(w_j,v_j)$ is wellposed on $[0,T]$. From perturbation theorem and separation of $x_{m,n}$,
we obtain
$$
\partial_tu({t_n} + t) = \sum\limits_{j = 1}^M {\partial_t{W_j}(x - {x_{j.n}}} ,t) + {o_{L^2}}(1).
$$
Since $\mathop {\lim }\limits_{n \to \infty } \int_0^T {\left\| {{\partial _t}u({t_n} + t)} \right\|_2^2} dt = 0$, by the separation of linear energy, we conclude
$$
\int_0^T {\left\| {{\partial _t}{W_j}(t)} \right\|_2^2} dt = 0.
$$
Therefore, $W_j$ is an equilibrium, the same holds for $w_j$, thus we have proved
there exists a finite number of equilibrium points $Q_m$ such that for any sequence $t_n\to\infty$, there exists $x_{m,n}$ for which
$$
u(t_n)=\sum\limits_{m = 1}^M Q_m(x-x_{m,n})+ {o_{H^1}}(1),  \mbox{  }\mbox{  }\partial_tu(t_n)={o_{L^2}}(1).
$$
By contradiction arguments, we can prove our theorem.

\end{document}